\documentclass[12pt,a4paper]{article}

\usepackage[utf8]{inputenc}
\usepackage[T1]{fontenc}
\usepackage[margin=1in]{geometry}
\usepackage{amsmath,amssymb,amsthm,mathtools}
\usepackage{enumitem}
\usepackage{cite}
\usepackage{hyperref}

\hypersetup{colorlinks=true,linkcolor=blue,citecolor=blue,urlcolor=blue}

\theoremstyle{plain}
\newtheorem{theorem}{Theorem}[section]
\newtheorem{proposition}[theorem]{Proposition}
\newtheorem{lemma}[theorem]{Lemma}
\newtheorem{corollary}[theorem]{Corollary}

\theoremstyle{definition}
\newtheorem{definition}[theorem]{Definition}
\newtheorem{example}[theorem]{Example}

\theoremstyle{remark}
\newtheorem{remark}[theorem]{Remark}

\newcommand{\HH}{\mathbb H}
\newcommand{\RR}{\mathbb R}
\newcommand{\NN}{\mathbb N}
\newcommand{\CC}{\mathbb C}
\newcommand{\BB}{\mathcal B}
\newcommand{\KK}{\mathcal K}
\newcommand{\Calk}{\mathcal C}
\newcommand{\Bor}{\mathfrak B}
\newcommand{\Id}{\mathrm{Id}}
\newcommand{\HS}{\mathcal H}
\newcommand{\sphere}{\mathbb S}
\newcommand{\dimH}{\dim_{\HH}}
\newcommand{\sigS}{\sigma_S}
\newcommand{\sigeS}{\sigma_{e}^{S}}
\newcommand{\sigwS}{\sigma_{W}^{S}}
\newcommand{\sigdS}{\sigma_{d}^{S}}
\newcommand{\sigeuS}{\sigma_{e+}^{S}}
\newcommand{\sigedS}{\sigma_{e-}^{S}}
\newcommand{\sigwuS}{\sigma_{w+}^{S}}
\newcommand{\sigwdS}{\sigma_{w-}^{S}}
\newcommand{\rhoS}{\rho_S}
\DeclareMathOperator{\ran}{ran}
\DeclareMathOperator{\ind}{ind}
\DeclareMathOperator{\dist}{dist}
\DeclareMathOperator{\supp}{supp}
\DeclareMathOperator{\Rea}{Re}
\DeclareMathOperator{\Ima}{Im}
\DeclareMathOperator{\essran}{ess\,ran}

\begin{document}
\par \vskip 0.3
cm \centerline{\bf  \large  Weyl's Theorem for Bounded Self-Adjoint Operators\\
  \\
}
 \par \vskip 0.3cm \centerline{\bf \large on Quaternionic Hilbert Spaces \large }

  \par
\vskip 0.7cm

 {\centerline{\large  Hatem Baloudi $^{\textbf{a}}$ Mohamed Ali Dbeibia $^{\textbf{b}}$ and Kamel Mahfoudhi $^{\textbf{c}}$}
 \par
\vskip 0.3cm
\centerline{\it \small $^{\textbf{a}}$Department of Mathematics, Faculty of Sciences of Gafsa,}
\centerline{\it\small  University
of Gafsa, 2112 Zarroug, Tunisia}
\centerline{e-mail: hatem.beloudi@gmail.com}
\centerline{\it \small $^{\textbf{b}}$ University of Sfax, Departement of Mathematics, Faculty of Sciences of Sfax}
\centerline{\it\small Route de soukra Km 3.5, B. P 1171, 3000, Sfax, Tunisia}
\centerline{e-mail: medalidbeibia@gmail.com}
\centerline{\it \small $^{\textbf{c}}$ University of Sousse, Higher Institute of Applied Sciences and Technology of Sousse,}
\centerline{\it\small 4003 Sousse Ibn Khaldoun, Sousse, Tunisia}
\centerline{e-mail: kamelmahfoudhi72@yahoo.com }

\begin{abstract}
Let $T$ be a bounded self-adjoint operator on a separable
infinite-dimensional right quaternionic Hilbert space, and let
$\sigdS(T)$ be the set of right eigenvalues of finite type. We prove that Weyl's theorem holds for $T$:
$$\sigeS(T)=\sigS(T)\setminus\sigdS(T),$$ which completes the inclusion
obtained in the general case. The proof proceeds through a single
three-way criterion: a point of $\sigS(T)$ lies outside $\sigeS(T)$,
and equally lies in $\sigdS(T)$, exactly when it is isolated with a
spectral atom of finite rank. The natural route through Riesz
projections is not available as it stands, because a Riesz projection
need not be orthogonal; we show that for a self-adjoint operator it in
fact is, the Riesz projection of an isolated spectral part coinciding
with the corresponding spectral projection, so that the algebraic and
geometric multiplicities agree. A $2\times2$ quaternionic matrix shows
that this last identity fails without self-adjointness, and a normal
operator with genuinely spherical $S$-spectrum shows what has to
change in the normal case.
\end{abstract}


\medskip


\section{Introduction}
\label{sec:intro}

The quaternions $\HH$ form a four-dimensional associative real
division algebra that extends the complex numbers; unlike $\CC$,
quaternionic multiplication is non-commutative, so the order of the
factors in a product $pq$ genuinely matters. This feature is exactly
what makes quaternions well suited to encoding rotations in
three-dimensional space, and it underlies their use in computer
graphics, robotics and aerospace engineering.

The spectral theory built on quaternionic operators is a comparatively
recent development, and it originates in an obstruction that has no
counterpart in the complex theory: if $T$ is a right linear operator
on a quaternionic Hilbert space, the map $$x\mapsto Tx-xq$$ fails to be
right linear as soon as $q\notin\RR$, so the classical resolvent
$T-q\Id$, and with it the classical notion of spectrum, simply does
not make sense for $q\in\HH\setminus\RR$. Motivated by quaternionic
quantum mechanics \cite{Adler1995}, Colombo and Sabadini resolved this
difficulty in 2006 by replacing $T-q\Id$ with the second-order,
real-coefficient pseudo-resolvent $$Q_q(T)=T^2-2\Rea(q)T+|q|^2\Id,$$
whose non-invertibility defines the $S$-spectrum $\sigS(T)$. A detailed
account of this history is given in
\cite[Section 1.2.1]{ColomboGantnerKimsey2018}. The associated slice
hyperholomorphic functional calculus was then developed, for bounded
and for unbounded operators, in a series of papers and monographs by
Colombo, Sabadini, Struppa and their collaborators
\cite{Colombo,ColomboSabadini2009,ColomboSabadini2010,
ColomboGentiliSabadiniStruppa2010,ColomboSabadiniStruppa2011,FJGanter,
GN0},
building on the theory of slice regular functions of Gentili and
Struppa \cite{GS}, related structural results in the same setting,
such as the Krein-Langer factorization for slice hyperholomorphic
functions, were obtained by Alpay, Colombo and Sabadini \cite{KRE}.
The Hilbert space theory proper came somewhat
later. Ghiloni, Moretti and Perotti constructed a continuous slice
functional calculus on quaternionic Hilbert spaces
\cite{GhiloniMorettiPerotti2013}, and the spectral theorem, with a
projection-valued measure supported on $\sigS(T)$, was proved first
for unitary operators by Alpay, Colombo, Kimsey and Sabadini
\cite{AlpayColomboKimsey2016Unitary}, and shortly after for bounded
and unbounded normal operators by Alpay, Colombo and Kimsey
\cite{AlpayColomboKimsey2016}; this material, together with the
$F$-functional calculus and spectral integration, is organized
systematically in the monograph \cite{ColomboGantnerKimsey2018}. Perturbation theory for normal
operators is developed in \cite{PF}, and the spectral theorem has
since been extended from the quaternionic to the more general Clifford
setting in \cite{FD1}. On the Fredholm side, Muraleetharan and
Thirulogasanthar introduced the essential $S$-spectrum through the
quaternionic Calkin algebra \cite{MuraleetharanThirulogasanthar2018}
and, building on it, the Weyl and Browder $S$-spectra
\cite{MuraleetharanThirulogasanthar2019}. Baloudi placed this Fredholm
theory within the broader framework of a homomorphism between
quaternionic Banach algebras \cite{BaloudiFredholm2022}, the
stability of the essential $S$-spectrum of
\cite{MuraleetharanThirulogasanthar2018} under a wider class of
perturbations is studied in \cite{Baloudi2020}, and isolated parts of
the $S$-spectrum and their associated spectral projections are
examined by Arzini and Jaatit \cite{ArziniJaatit2025}. It is against
this background that Baloudi, Belgacem and Jeribi took up the
quaternionic Weyl problem \cite{BaloudiBelgacemJeribi2022}, which is
the starting point of the present paper.

For a bounded self-adjoint operator on a complex Hilbert space, Weyl's
theorem \cite{Weyl1909} states that
$$\sigma_e(T)=\sigma(T)\setminus\sigma_d(T),$$ where $\sigma_e(T)$
denotes the essential spectrum of $T$ and $\sigma_d(T)$ its discrete
spectrum, i.e., its isolated eigenvalues of finite multiplicity. The
theorem extends well beyond the self-adjoint case, see \cite{COB} for
its original extension to non-normal operators, \cite{HAR} for a
systematic treatment through Fredholm and Weyl operators, and
\cite{OUD} for its stability under perturbation.

In \cite{BaloudiBelgacemJeribi2022}, Baloudi, Belgacem and Jeribi
introduced the discrete $S$-spectrum $\sigdS(T)$, the right
eigenvalues of finite type, defined by a decomposition of the space
(Definition \ref{def:finite-type}) and proved one half of Weyl's
theorem: for every bounded right linear $T$,
\begin{equation}
 \sigeS(T)\subset\sigS(T)\setminus\sigdS(T).
 \label{eq:inclusion}
\end{equation}
The reverse inclusion remained open, the obstruction being that
$S$-spectral points come in spherical classes
$[q]=\{\Rea(q)+I|\Ima(q)|:I\in\sphere\}$ rather than singletons.

The self-adjoint case is not merely the easiest one. Observables, and
in particular Hamiltonians, are self-adjoint operators in quaternionic
quantum mechanics \cite{Adler1995}, and Weyl's theorem is precisely
what separates the bound states of a system, isolated energy levels
of finite degeneracy from the scattering continuum. In particular,
Theorem \ref{thm:main} shows that the essential $S$-spectrum is
invariant under compact self-adjoint perturbations, expressing the
corresponding stability of the continuum.
\begin{center}
Where does the difficulty actually lie?
\end{center}

Self-adjointness places the $S$-spectrum inside $\RR$, so each
$S$-spectral set reduces to a singleton $[\lambda]=\{\lambda\}$,
which invites an immediate conclusion. The natural line of reasoning is
as follows: suppose \(\lambda\in\sigS(T)\setminus\sigeS(T)\). Then
\(T-\lambda\Id\) is Fredholm and \(\lambda\) is isolated, so the Riesz
projection \(P_{\{\lambda\}}\) is well defined. On
\(\ran(P_{\{\lambda\}})\), the operator \(T\) restricts to a
self-adjoint operator whose \(S\)-spectrum is \(\{\lambda\}\), such a
restriction must coincide with \(\lambda\Id\). Consequently
\(\ran(P_{\{\lambda\}})\subset\ker(T-\lambda\Id)\), a finite-dimensional
space, and therefore \(\lambda\in\sigdS(T)\) by
\cite[Theorem 3.10]{BaloudiBelgacemJeribi2022}. This argument is not valid as it stands. A Riesz projection need
not be orthogonal, so $\ran(P_{\{\lambda\}})$ need not reduce $T$ and
$T|_{\ran(P_{\{\lambda\}})}$ need not be self-adjoint and the
inclusion $\ran(P_{\{\lambda\}})\subset\ker(T-\lambda\Id)$ genuinely
fails in general, as Example \ref{ex:jordan} shows. The present paper
resolves this in two independent ways.

\begin{enumerate}
\item \textbf{Weyl's theorem} (Theorem \ref{thm:main}), proved without
Riesz projections at all. Everything follows from the single criterion
of Proposition \ref{prop:criterion}: for $\lambda\in\sigS(T)$, the
statements $\lambda\notin\sigeS(T)$ and $\lambda\in\sigdS(T)$ are
both equivalent to
\[
 \lambda\ \text{isolated in}\ \sigS(T)
 \quad\text{and}\quad
 \dimH\ker(T-\lambda\Id)<\infty .
\]
Each equivalence is obtained directly from the spectral measure, the
decomposition $$\HS=\ran(E(\{\lambda\}))\oplus\ker(E(\{\lambda\}))$$
being orthogonal by construction. Only one of the four implications
carries the argument, and Remark \ref{rem:where-the-work-is} isolates
it: it is the one that produces the decomposition required by the
definition of $\sigdS$, and it is unavailable without Lemma
\ref{lem:restriction}. The proof is therefore independent
of \eqref{eq:inclusion} and of
\cite[Theorem 3.10]{BaloudiBelgacemJeribi2022}.

\item \textbf{Riesz projections are orthogonal here} (Theorem
\ref{thm:riesz}). For $T=T^*$ and $\Delta$ an isolated part of
$\sigS(T)$ one has $P_\Delta=E(\Delta)$, in particular
$P_\Delta^*=P_\Delta$, and
\[
 m_T(\lambda)=\dimH\ran\big(P_{\{\lambda\}}\big)
 =\dimH\ker(T-\lambda\Id)
\]
for isolated $\lambda$. This supplies exactly the step missing above,
so it also repairs the Riesz-projection argument and it identifies
the algebraic multiplicity of \cite{BaloudiBelgacemJeribi2022} with
the geometric one.
\end{enumerate}

What makes the self-adjoint case work is thus not that spherical
classes degenerate to points that only simplifies bookkeeping
but that the spectral measure furnishes an orthogonal
decomposition adapted to an isolated spectral part. Example
\ref{ex:sphere} bears this out. It is a normal operator with genuinely
spherical $S$-spectrum for which the conclusion still holds, and it
locates the obstacle in the normal case, namely that $Q_q(T)$ is no
longer a square for $q\notin\RR$, so that Lemma \ref{lem:square},
which underlies the Fredholm criterion of Lemma \ref{lem:fredholm},
does not apply.

Section \ref{sec:prelim} fixes hypotheses and notation, together with
the Fredholm classes and the various essential and Weyl
$S$-spectra. Section \ref{sec:fredholm} establishes two algebraic
facts about the $S$-spectrum used later, invertibility via squares
(Lemma \ref{lem:square}), and the coincidence, once $T$ is
self-adjoint, of the essential and Weyl $S$-spectra together with
their upper and lower variants (Lemma \ref{lem:hierarchy}). Section
\ref{sec:spectral} recalls the spectral
theorem and establishes the two lemmas on spectral projections on
which everything rests. Section \ref{sec:main} proves the criterion
$(1)$, deduces Weyl's theorem and closes with a characterization in
terms of singular sequences, the form in which the essential
$S$-spectrum is used in practice. Section \ref{sec:riesz} proves the
complement $(2)$. Section \ref{sec:examples} contains the examples and
Section \ref{sec:conclusion} two open problems.

\section{Preliminaries}
\label{sec:prelim}

\subsection{Quaternions}
We will use $\mathbb{H}$ to denote the Hamiltonian skew field of quaternions, which has a standard basis $\{1,i,j,k\}$. It is formally defined as: $$\mathbb{H} = \{q = x_0 + x_1i + x_2j + x_3k : x_i \in \mathbb{R}, i = 0, 1, 2, 3\}.$$The three imaginary units $i,j,k$ satisfy the following  relations:$$i^2 = j^2 = k^2 = ijk = -1, \quad ij = -ji = k, \quad ki = -ik = j, \quad jk = -kj = i.$$
For a quaternion $q = x_0 + x_1i + x_2j + x_3k \in \mathbb{H}$, its real part is $\text{Re}(q) = x_0$ and its imaginary part is $\text{Im}(q) = x_1i + x_2j + x_3k$. The conjugate and the norm of $q$ are then defined as:
$$\overline{q} = \text{Re}(q) - \text{Im}(q) \quad \text{and} \quad |q| = \sqrt{x_0^2 + x_1^2 + x_2^2 + x_3^2}.$$
The unit sphere of purely imaginary quaternions is $\mathbb{S}$, given by:
$$\mathbb{S} = \{q \in \mathbb{H} : \text{Re}(q) = 0 \text{ and } \overline{q}q = 1\}.$$
It is important to note that $\mathbb{S}$ is a two-dimensional sphere in $\mathbb{R}^4$. For any quaternion $q\in\mathbb{H}\setminus\mathbb{R}$, we can write it as:

$$q = \text{Re}(q) + I_q|\text{Im}(q)|,$$where $I_q = \frac{\text{Im}(q)}{|\text{Im}(q)|} \in \mathbb{S}$. This allows us to associate a two-dimensional sphere with $q$, defined as:$$[q] = \text{Re}(q) + \mathbb{S}|\text{Im}(q)|.$$
This sphere, $[q]$, is centered at the real point ${\rm Re}(q)$ and has a radius  $|Im(q)|$. This sphere is equivalent to the set
$\{hqh^{-1}:\ h\in \mathbb{H}^*\}$. For the full proof, refer to
\cite[Chapter 2]{ColomboGantnerKimsey2018}.
Finally, for any $I\in\mathbb{S}$, we set:

$$\mathbb{C}_I = \mathbb{R} + I\mathbb{R}.$$From this, it follows that:$$\mathbb{H} = \bigcup_{I \in \mathbb{S}} \mathbb{C}_I.$$

\subsection{Hypotheses}

\begin{itemize}[leftmargin=3em]
\item[(A1)] $\HS$ is a right quaternionic Hilbert space: a right
$\HH$-module with a quaternion-valued inner product, right linear in
its second argument, hermitian and positive definite, complete for
$\|x\|=\langle x,x\rangle^{1/2}$. We assume $\HS$ separable; this is
used only through the spectral theorem
and through the extraction of orthonormal sequences in Theorem
\ref{thm:weyl-criterion}. We write $\BB(\HS)$ for the real unital
Banach algebra of bounded right linear operators, $\KK(\HS)$ for the
closed ideal of compact ones, $T^*$ for the adjoint.

\item[(A2)] From Section \ref{sec:fredholm} on, $\dimH\HS=\infty$, so
the Calkin algebra $\Calk(\HS)=\BB(\HS)/\KK(\HS)$ is a nonzero real
unital Banach algebra; $\pi$ denotes the quotient map.

\item[(A3)] Section \ref{sec:riesz} only, $\HS$ carries in
addition a fixed left scalar multiplication $q\mapsto L_q$ with
$L_q\in\BB(\HS)$, $L_qL_{q'}=L_{qq'}$ and $L_q^*=L_{\overline q}$.
This is what gives a meaning to $T-\overline s\,\Id$ for $s\notin\RR$,
hence to the $S$-resolvent operators; it is the setting of
\cite{ColomboGantnerKimsey2018,BaloudiBelgacemJeribi2022}. Hypothesis
(A3) costs nothing, once a Hilbert basis $(e_n)$ of $\HS$ is fixed,
\[
 L_q\Big(\sum_n e_nx_n\Big)=\sum_n e_n\,qx_n
\]
defines such a family, each $L_q$ being isometric with
$L_qL_{q'}=L_{qq'}$ and $L_q^*=L_{\overline q}$; this construction is
standard, see \cite[Section 3]{GhiloniMorettiPerotti2013}. The family
depends on the chosen basis, but nothing below does: only the
existence of one such family is used. We state it separately
only because Sections \ref{sec:fredholm}--\ref{sec:main} do not use
it. Since $Q_q(\cdot)$ has real coefficients, everything before
Section \ref{sec:riesz} needs (A1)--(A2) only.
\end{itemize}

To avoid a collision that is easy to make in this
subject, we fix once and for all:
\[
 \begin{array}{ll}
 I,\ I':&\text{imaginary units, i.e.\ elements of }\sphere;
 \text{ they occur only in }\CC_I\text{ and }ds_I,\\[2pt]
 \Id:&\text{the identity operator on }\HS,\\[2pt]
 q,\ s:&\text{quaternionic spectral parameters},\\[2pt]
 \lambda,\ t:&\text{real spectral parameters}.
 \end{array}
\]
The letter $I$ never denotes an operator. $\dimH W$ is the
cardinality of an orthonormal basis of the closed right submodule $W$,
the real dimension being $4\dimH W$, finiteness is unaffected. We write
$$\ind(A)=\dimH\ker(A)-\dimH\ker(A^*)$$ for the index of a Fredholm
operator.

\subsection{Different Types of S-Spectra}
For $T\in \BB(\HS)$ and $q\in\mathbb{H}$, we define the associated operator $\mathcal{Q}_q(T):\ \HS\longrightarrow \HS$ by setting
\begin{align*}Q_{q}(T):=T^2-2{\rm Re}(q)T+\vert q\vert^2Id.\end{align*}
Let $T\in \BB(\HS)$. The $S$-spectrum of $T$ is defined as
\begin{align*}\sigma_{S}(T)=\Big\{q\in\mathbb{H}:\ Q_{q}(T)\mbox{ is not invertible in } \mathcal{B}(V^{R}_{\mathbb{H}})\Big\}.\end{align*}
Then, we define the $S$-resolvent set of $T$ as
\begin{align*}\rho_{S}(T)=\mathbb{H}\setminus\sigma_{S}(T).\end{align*}
For a complete explanation, we refer to \cite{ColomboGantnerKimsey2018}. Note that $\sigma_{S}(T)$ is a non-empty compact set, as shown in \cite{FJGanter}. If the relation $Tu=uq$ holds for some non-zero $u\in \HS$ and a quaternion $q\in\mathbb{H}$, then $u$ is called an eigenvector of $T$ with the right eigenvalue $q$.
\begin{definition} \label{def:axially-symmetric}
{\rm A set $\Omega\subset\mathbb{H}$ is called\\
\noindent $(i)$ axially symmetric if $\{hqh^{-1}:\ h\in\mathbb{H}\}\subset \Omega$ for any $q\in \Omega$ and\\
\noindent $(ii)$ a slice domain (or $s$-domain for short) if $\Omega$ is open, $\Omega\cap\mathbb{R}\neq\emptyset$ and $\Omega\cap\mathbb{C}_{I}$ is a domain in $\mathbb{C}_{I}$, for any $I\in\mathbb{S}$.}
\end{definition}

\noindent The $S$-spectrum $\sigma_{S}(T)$ and the $S$-resolvent set $\rho_{S}(T)$ are both axially symmetric. See \cite{ColomboGantnerKimsey2018} for further details. For $\lambda\in\RR$,
$Q_\lambda(T)=(T-\lambda\Id)^2$.
\begin{definition}[{\cite[Definition 3.7]{BaloudiBelgacemJeribi2022}}]
\label{def:finite-type}
Let $T\in\BB(\HS)$. A point $q\in\sigS(T)$ is a right eigenvalue
of finite type if $\HS=\HS_1\oplus\HS_2$ is a topological direct sum
into closed $T$-invariant right submodules with
\begin{enumerate}[label=(\textup{F}\arabic*),leftmargin=3em]
\item $\dimH\HS_1<\infty$,
\item $\sigS(T|_{\HS_1})\cap\sigS(T|_{\HS_2})=\varnothing$,
\item $\sigS(T|_{\HS_1})=[q]$.
\end{enumerate}
The set of such $q$ is the discrete $S$-spectrum $\sigdS(T)$.
\end{definition}
From this point onward, we assume (A2) holds. Denote by \(\Phi(\HS)\) the set of
Fredholm operators, i.e., those bounded operators having closed range and
finite-dimensional kernel and cokernel: \(\dimH\ker(A)<\infty\) and
\(\dimH\ker(A^*)<\infty\). Within this class, let
$$\Phi_0(\HS)=\{A\in\Phi(\HS):\ind(A)=0\}$$ consist of the Fredholm operators of
index zero. The essential and Weyl \(S\)-spectra of \(T\) are then
defined by
\[
 \sigeS(T)=\{q\in\HH:\ Q_q(T)\notin\Phi(\HS)\},
 \qquad
 \sigwS(T)=\{q\in\HH:\ Q_q(T)\notin\Phi_0(\HS)\}.
\]
As in the classical theory, $\Phi(\HS)$ splits into the two
semi-Fredholm classes
$$\Phi_+(\HS)=\{A\in\BB(\HS):\ \ran(A)\ \text{closed},\
 \dimH\ker(A)<\infty\},$$
$$ \Phi_-(\HS)=\{A\in\BB(\HS):\ \ran(A)\ \text{closed},\
 \dimH\ker(A^*)<\infty\},$$
so that $\Phi(\HS)=\Phi_+(\HS)\cap\Phi_-(\HS)$
\cite{MuraleetharanThirulogasanthar2018}, and correspondingly $\sigeS$
and $\sigwS$ each split into an upper and a lower part.The
upper and lower essential $S$-spectra
\[
 \sigeuS(T)=\{q\in\HH:\ Q_q(T)\notin\Phi_+(\HS)\},
 \qquad
 \sigedS(T)=\{q\in\HH:\ Q_q(T)\notin\Phi_-(\HS)\},
\]
so that $\sigeS(T)=\sigeuS(T)\cup\sigedS(T)$, and the upper and lower Weyl $S$-spectra
$$\sigwuS(T)=\{q\in\HH:\ Q_q(T)\notin\Phi_+(\HS)\ \text{or}\
 \ind(Q_q(T))>0\},$$
 $$\sigwdS(T)=\{q\in\HH:\ Q_q(T)\notin\Phi_-(\HS)\ \text{or}\
 \ind(Q_q(T))<0\},$$
so that likewise $\sigwS(T)=\sigwuS(T)\cup\sigwdS(T)$. These four
spectra, together with the Browder $S$-spectrum, are studied for
general $T$ in
\cite{MuraleetharanThirulogasanthar2018,MuraleetharanThirulogasanthar2019},
they are not needed in general here, but Lemma \ref{lem:hierarchy}
below shows that all six of
$\sigeS,\sigwS,\sigeuS,\sigedS,\sigwuS,\sigwdS$ collapse to a single
set once $T$ is self-adjoint. By Atkinson's theorem in the
quaternionic setting
\cite{MuraleetharanThirulogasanthar2018}, $Q_q(T)\in\Phi(\HS)$ if and
only if $\pi(Q_q(T))$ is invertible in $\Calk(\HS)$, so that
\begin{equation}
 \sigeS(T)=\sigS\big(\pi(T)\big),
 \label{eq:calkin}
\end{equation}
which is the definition adopted in
\cite{Baloudi2020,MuraleetharanThirulogasanthar2018}; the two
conventions therefore agree. Since $\Phi_0(\HS)\subset\Phi(\HS)$ we
have $\sigeS(T)\subseteq\sigwS(T)$, and since $q\in\rhoS(T)$ makes
$\pi(Q_q(T))$ invertible we have $\sigeS(T)\subset\sigS(T)$.

\section{Some lemmas on the $S$-spectrum and the essential $S$-spectrum}
\label{sec:fredholm}

We begin by stating a basic algebraic observation that will be invoked
several times in what follows, allowing us to relate $\sigS(T)$ to the
$S$-spectra of restrictions of $T$. Recall that $\HH$ is a
four-dimensional associative real division algebra, and that, upon
fixing a Hilbert basis of $\HS$ (hypothesis (A3)), the space
$\BB(\HS)$ acquires the structure of a quaternionic two-sided
Banach algebra, that is, a real Banach algebra carrying left and right
$\HH$-actions, $$q\mapsto qA ~~\hbox{and}~~ q\mapsto Aq,$$ which are compatible
with composition in the sense that $$q(AB)=(qA)B ~~\hbox{and}~~(AB)q=A(Bq).$$

This is the structure underlying the Fredholm theory relative to a
homomorphism of quaternionic Banach algebras developed in
\cite{BaloudiFredholm2022}, which carries the classical theory of
Fredholm elements relative to a Banach algebra homomorphism over to the
quaternionic setting. Lemma \ref{lem:square} below, however, requires
none of this two-sided structure, it remains valid in any unital,
possibly non-commutative, real algebra, whether quaternionic or not.
This is precisely what allows it to be applied to $\BB(\HS)$ viewed
merely as a real algebra, without appealing to hypothesis (A3)
consistent with the observation in Section \ref{sec:prelim} that
Sections \ref{sec:fredholm}-\ref{sec:main} require only (A1)-(A2).
\begin{lemma}
\label{lem:square}
In a real unital algebra, $a$ is invertible if and only if $a^2$ is.
Consequently, if $\HS=\HS_1\oplus\HS_2$ is a topological direct sum
into closed $T$-invariant right submodules, then
$$Q_q(T)=Q_q(T|_{\HS_1})\oplus Q_q(T|_{\HS_2}) ~~\hbox{and}~~ \sigS(T)=\sigS(T|_{\HS_1})\cup\sigS(T|_{\HS_2}).$$
\end{lemma}

\begin{proof}
For the first assertion, only the nontrivial implication requires proof,
since the converse is immediate. If $a$ is invertible, then so is
$$a^2=a\cdot a,$$ being a product of invertible elements, with inverse
$$a^{-1}a^{-1}.$$ Conversely, suppose $a^2$ is invertible. Since $a$
commutes with $a^2$, it also commutes with $$(a^2)^{-1},$$ hence
$$b=a(a^2)^{-1}=(a^2)^{-1}a$$ satisfies $ab=ba=1$. For the second
assertion, every real polynomial in $T$ preserves both summands, and an
operator that respects a topological direct sum is invertible precisely
when both of its restrictions are.
\end{proof}

\begin{proposition}
\label{prop:real-spectrum}
If $T\in\BB(\HS)$ is self-adjoint, then $\sigS(T)\subset\RR$.
\end{proposition}
\begin{proof}
We refer also to \cite[Theorem 2.11(ii)]{AlpayColomboKimsey2016} for
this fact. One has
$$Q_q(T)=(T-\Rea(q)\Id)^2+|\Ima(q)|^2\Id.$$ Given $q\notin\RR$, set
$B=Q_q(T)$ and $c=|\Ima(q)|^2>0$. Then $$B=B^* ~~\hbox{and}~~
\langle Bx,x\rangle\ge c\|x\|^2,$$ so Cauchy-Schwarz yields
$$\|Bx\|\ge c\|x\|.$$ Consequently $B$ is injective with closed range,
while $\ran(B)^\perp=\ker(B)=\{0\}$ shows that it is also surjective.
\end{proof}
\begin{lemma}
\label{lem:hierarchy}
If $T=T^*\in\BB(\HS)$, then
\[
 \sigwS(T)=\sigeS(T)=\sigeuS(T)=\sigedS(T)=\sigwuS(T)=\sigwdS(T).
\]
\end{lemma}

\begin{proof}
The inclusion $\sigeS(T)\subseteq\sigwS(T)$ was already noted.
Conversely, let $q\notin\sigeS(T)$, so that $A=Q_q(T)$ is Fredholm; as
$Q_q$ has real coefficients and $T=T^*$, we have $A=A^*$, hence
$$\ker(A)=\ker(A^*) ~~\hbox{and} ~~\ind(A)=0.$$ Thus $$A\in\Phi_0(\HS)~~\hbox{and} ~~q\notin\sigwS(T),$$ this proves $$\sigwS(T)=\sigeS(T).$$
The same observation, namely that $A=Q_q(T)=A^*$ for every $q\in\HH$
whenever $T=T^*$, disposes of the remaining four spectra in one go.
First, since $\ker(A)=\ker(A^*)$, the two conditions defining
$\Phi_+(\HS)$ and $\Phi_-(\HS)$ closed range together with
$\dimH\ker(A)<\infty$ and $\dimH\ker(A^*)<\infty$, respectively,
become identical, so $A\in\Phi_+(\HS)$ precisely when
$A\in\Phi_-(\HS)$, equivalently $$\sigeuS(T)=\sigedS(T),$$ and because
$\sigeS(T)=\sigeuS(T)\cup\sigedS(T)$ by definition, the three sets
coincide $$\sigeuS(T)=\sigedS(T)=\sigeS(T).$$ Second, whenever
$Q_q(T)\in\Phi_+(\HS)$ and therefore $Q_q(T)\in\Phi_-(\HS)$ by the
foregoing, so that $Q_q(T)\in\Phi(\HS)$ the identity
$\ker(A)=\ker(A^*)$ entails
$$\ind(A)=\dimH\ker(A)-\dimH\ker(A^*)=0.$$ Consequently the extra clauses
``$\ind(Q_q(T))>0$'' and ``$\ind(Q_q(T))<0$'' appearing in the
definitions of $\sigwuS(T)$ and $\sigwdS(T)$ are never activated once
$Q_q(T)\in\Phi_+(\HS)$, respectively $Q_q(T)\in\Phi_-(\HS)$, thus
$$\sigwuS(T)=\sigeuS(T) ~~\hbox{and}~~ \sigwdS(T)=\sigedS(T).$$ Together with the
first part, this yields
$\sigwuS(T)=\sigwdS(T)=\sigeS(T)=\sigwS(T)$, and the chain of
equalities is complete.
\end{proof}

\section{The spectral theorem and spectral projections}
\label{sec:spectral}

Throughout this section, $T=T^*\in\BB(\HS)$. By Proposition
\ref{prop:real-spectrum}, $\sigS(T)$ is a nonempty compact subset of
$\RR$, and $[\lambda]=\{\lambda\}$ for each $\lambda\in\sigS(T)$.
Recall that $\Bor(\RR)$ denotes the Borel $\sigma$-algebra of $\RR$,
and that a projection-valued measure on $(\RR,\Bor(\RR))$ with
values in $\BB(\HS)$ is a map $E:\Bor(\RR)\to\BB(\HS)$ satisfying the
following: each $E(B)$ is an orthogonal projection, $$E(\RR)=\Id,$$
$$E(B_1\cap B_2)=E(B_1)E(B_2)$$ for all $B_1,B_2\in\Bor(\RR)$ and $E$
is countably additive in the strong operator topology, that is,
$$E\big(\bigcup_nB_n\big)x=\sum_nE(B_n)x$$ for every $x\in\HS$ and every
sequence of pairwise disjoint sets $(B_n)\subset\Bor(\RR)$. Its
support $\supp(E)$ is defined as the smallest closed
$F\subset\RR$ with $E(\RR\setminus F)=0$, or equivalently as the
complement of the union of all open $O\subset\RR$ with $E(O)=0$.

\begin{theorem}[{Spectral theorem; \cite[Corollary~11.2.2(ii)]{ColomboGantnerKimsey2018},
see also \cite{GhiloniMorettiPerotti2013,AlpayColomboKimsey2016}}]
Let $T=T^*\in\BB(\HS)$. There is a unique projection-valued measure
$E:\Bor(\RR)\to\BB(\HS)$, with values orthogonal projections, such
that $$T=\int_\RR t\,dE(t),$$ it satisfies $\supp(E)=\sigS(T)$. For
every bounded Borel $f:\RR\to\RR$ the operator $$f(T)=\int_\RR f\,dE$$
is bounded self-adjoint, $f\mapsto f(T)$ is a unital homomorphism of
real algebras, and
\begin{equation}
 \|f(T)x\|^2=\int_\RR|f|^2\,d\mu_x,
 \qquad
 \mu_x(B)=\langle E(B)x,x\rangle=\|E(B)x\|^2 .
 \label{eq:scalar}
\end{equation}
\end{theorem}
Each $\mu_x$ is a finite positive real measure, $E(B)$ being an
orthogonal projection.

\begin{lemma}
\label{lem:eigenspace}
For every $\lambda\in\RR$, $\ran(E(\{\lambda\}))=\ker(T-\lambda\Id)$.
\end{lemma}

\begin{proof}
Let $E(\{\lambda\})x=x$ then
$$\mu_x(\RR\setminus\{\lambda\})=\|x-E(\{\lambda\})x\|^2=0,$$ so
\eqref{eq:scalar} with $f(t)=t-\lambda$ gives
$$\|(T-\lambda\Id)x\|^2=0.$$ Conversely $(T-\lambda\Id)x=0$ gives
$$\int|t-\lambda|^2d\mu_x=0.$$ The integrand is nonnegative and vanishes
only at $\lambda$, so $$\mu_x(\RR\setminus\{\lambda\})=0 ~~\hbox{and}~~
\|x-E(\{\lambda\})x\|^2=\|E(\RR\setminus\{\lambda\})x\|^2=0.$$
\end{proof}

The next lemma is the technical core of the paper.

\begin{lemma}
\label{lem:restriction}
Let $\Delta\subset\sigS(T)$ be nonempty, compact and open in
$\sigS(T)$, and put $$\Delta'=\sigS(T)\setminus\Delta,
\HS_1=\ran(E(\Delta)), \HS_2=\ker(E(\Delta)).$$ Then $\HS_1,\HS_2$
are closed, mutually orthogonal, $T$-invariant,
$\HS=\HS_1\oplus\HS_2$, and
\[
 \sigS(T|_{\HS_1})=\Delta,\qquad \sigS(T|_{\HS_2})=\Delta' .
\]
Moreover $\HS_1\neq\{0\}$, and
$\dist(\sigS(T|_{\HS_1}),\sigS(T|_{\HS_2}))>0$ when
$\Delta'\neq\varnothing$.
\end{lemma}

\begin{proof}
The operator \(E(\Delta)\) is an orthogonal projection commuting with \(T\). Hence
\[
\HS_1=\ran E(\Delta),\qquad \HS_2=\ran E(\Delta')
\]
are closed, orthogonal, \(T\)-invariant, and \(\HS=\HS_1\oplus\HS_2\). Since \(\Delta\) is compact and closed in \(\sigS(T)\), its complement \(\Delta'\) is compact and open in \(\sigS(T)\). Also \(E\) is carried by $$\supp(E)=\sigS(T),$$ so
\[
E(\RR\setminus\Delta)=E(\Delta'),
\]
and therefore \(\HS_2=\ran(E(\Delta'))\). For \(q\in\CC\), define
\[
\varphi_q(t)=t^2-2\Rea(q)t+|q|^2
=\bigl(t-\Rea(q)\bigr)^2+|\Ima(q)|^2 .
\]
Then \(Q_q(T)=\varphi_q(T)\) in the Borel functional calculus. We treat \(\Delta\), the argument for \(\Delta'\) is identical. If \(\Delta'=\varnothing\), then \(\HS_2=\{0\}\) and \(\sigS(T|_{\HS_2})=\varnothing\).

\emph{Step 1: \(q\notin\Delta\) implies \(q\in\rhoS(T|_{\HS_1})\).}
If \(q\notin\RR\), then \(\varphi_q\ge |\Ima(q)|^2>0\). If \(q\in\RR\) and \(q\notin\Delta\), then
\[
\varphi_q(t)=(t-q)^2\ge \dist(q,\Delta)^2>0
\]
on the compact set \(\Delta\). Thus in either case
\[
\delta:=\inf_{\Delta}\varphi_q>0.
\]
Hence \(\psi:=\varphi_q^{-1}\chi_\Delta\) is bounded Borel with \(\|\psi\|_\infty\le\delta^{-1}\), and
\[
\varphi_q\psi=\chi_\Delta
\]
gives
\[
Q_q(T)\,\psi(T)=\psi(T)\,Q_q(T)=E(\Delta).
\]
Both operators preserve \(\HS_1\), and \(E(\Delta)|_{\HS_1}=\Id_{\HS_1}\). Therefore \(\psi(T)|_{\HS_1}\) inverts \(Q_q(T|_{\HS_1})\), so \(q\in\rhoS(T|_{\HS_1})\).

\emph{Step 2: \(q\in\Delta\) implies \(q\in\sigS(T|_{\HS_1})\).}
Such a \(q\) is real. Let \(\varepsilon>0\) and
\[
B_\varepsilon=(q-\varepsilon,q+\varepsilon)\cap\Delta .
\]
Since \(\Delta\) is only relatively open in \(\sigS(T)\), write \(\Delta=W\cap\sigS(T)\) with \(W\subset\RR\) open. Then
\[
O=(q-\varepsilon,q+\varepsilon)\cap W
\]
is open in \(\RR\) and contains \(q\in\supp(E)\). Hence \(E(O)\neq0\). Moreover,
\[
E(O)=E(O\cap\sigS(T))=E(B_\varepsilon),
\]
because \(E\) is carried by \(\supp(E)\). Choose a unit vector
\[
x\in\ran(E(B_\varepsilon))\subset\HS_1 .
\]
Then \(\mu_x\) is carried by \(B_\varepsilon\), and on \(B_\varepsilon\),
\[
|\varphi_q(t)|=(t-q)^2<\varepsilon^2 .
\]
Thus
\[
\|Q_q(T|_{\HS_1})x\|^2\le\varepsilon^4
\]
by \eqref{eq:scalar}. Since \(\varepsilon>0\) is arbitrary, \(Q_q(T|_{\HS_1})\) is not bounded below. Hence \(q\in\sigS(T|_{\HS_1})\).

Steps 1 and 2 yield
\[
\sigS(T|_{\HS_1})=\Delta .
\]
Applying the same argument to \(\Delta'\) gives
\[
\sigS(T|_{\HS_2})=\Delta' .
\]
Since \(\Delta\neq\varnothing\), Step 2 supplies a unit vector in \(\HS_1\), so \(\HS_1\neq\{0\}\). Finally, \(\Delta\) and \(\Delta'\) are disjoint compact sets.
\end{proof}

\section{Weyl's theorem}
\label{sec:main}

\emph{Here $T=T^*\in\BB(\HS)$, $\dimH\HS=\infty$, and $E$ is the
spectral measure of $T$.}

\begin{lemma}
\label{lem:fredholm}
For $\lambda\in\RR$: $\lambda\notin\sigeS(T)$ if and only if
$T-\lambda\Id$ is Fredholm. If $T-\lambda\Id$ is Fredholm, then either
$\lambda\in\rhoS(T)$ or $\lambda$ is an isolated point of $\sigS(T)$.
\end{lemma}

\begin{proof}
Since \(Q_\lambda(T)=(T-\lambda\Id)^2\), we have
\(\pi(Q_\lambda(T))=\pi(T-\lambda\Id)^2\). By Lemma \ref{lem:square},
the first of these is invertible in \(\Calk(\HS)\) exactly when the
second is. Combining \eqref{eq:calkin} with Atkinson's theorem
\cite{MuraleetharanThirulogasanthar2018}, this gives the first
assertion.

For the second assertion, set \(A=T-\lambda\Id=A^*\). Because \(A\) is
Fredholm, its kernel is finite-dimensional and its range is closed,
moreover
$$\ran(A)=\ker(A^*)^\perp=\ker(A)^\perp.$$ Since
\(A(\HS)=A(\ker(A)^\perp)\), it follows that the restricted operator
\(A_0=A|_{\ker(A)^\perp}\) maps \(\ker(A)^\perp\) bijectively onto
itself, and therefore \(A_0\) is boundedly invertible.

Let \(c=\|A_0^{-1}\|^{-1}>0\). If \(t\in\RR\) satisfies \(|t|<c\), then
$$A_0-t\Id=A_0(\Id-tA_0^{-1})$$ is invertible by the Neumann series.
Hence \(Q_t(A_0)=(A_0-t\Id)^2\) is invertible, so
$$t\notin\sigS(A_0).$$ On the other hand, since \(A|_{\ker A}=0\), we
have $$\sigS(A|_{\ker A})\subset\{0\}.$$ Applying Lemma
\ref{lem:square}, we obtain
\[
\sigS(A)\subset\{0\}\cup\bigl(\RR\setminus(-c,c)\bigr),
\]
so \(0\) is either outside \(\sigS(A)\) or isolated in it. Finally, \(A\) and \(T\) are self-adjoint, so their \(S\)-spectra are
real. Since \(Q_t(A)=Q_{\lambda+t}(T)\) for real \(t\), it follows
that $$\sigS(A)=\sigS(T)-\lambda.$$
\end{proof}

\begin{proposition}
\label{prop:criterion}
Let $\lambda\in\sigS(T)$. The following are equivalent:
\begin{enumerate}[label=(\roman*),leftmargin=2.4em]
\item $\lambda\notin\sigeS(T)$,
\item $\lambda\in\sigdS(T)$,
\item $\lambda$ is isolated in $\sigS(T)$ and
$\dimH\ker(T-\lambda\Id)<\infty$,
\item $\lambda$ is isolated in $\sigS(T)$ and the atom
$E(\{\lambda\})$ has finite rank.
\end{enumerate}
\end{proposition}

\begin{proof}
We take (iii) as the central statement and prove
(iii) $\Leftrightarrow$ (iv), (iii) $\Leftrightarrow$ (i) and
(iii) $\Leftrightarrow$ (ii). The argument relies only on Lemma
\ref{lem:square}, Lemma \ref{lem:eigenspace}, Lemma
\ref{lem:restriction} and Lemma \ref{lem:fredholm}, each of which is
established in Section \ref{sec:prelim}, \ref{sec:spectral} or above
without reference to the present proposition. Consequently, none of
the implications below is used in the proof of any statement on which
it depends.

The equivalence (iii) $\Leftrightarrow$ (iv) is precisely Lemma
\ref{lem:eigenspace}.

Suppose now that (iii) holds, and set
$$\HS_1=\ran(E(\{\lambda\})), \HS_2=\ker(E(\{\lambda\})).$$ Since
$\lambda$ is isolated, Lemma \ref{lem:restriction} applies with
$\Delta=\{\lambda\}$ and yields
$\sigS(T|_{\HS_1})=\{\lambda\}=[\lambda]$ together with
$\sigS(T|_{\HS_2})=\sigS(T)\setminus\{\lambda\}$. These two $S$-spectra
are disjoint, furthermore $\HS_1=\ker(T-\lambda\Id)$ by Lemma
\ref{lem:eigenspace}, it is nonzero by Lemma \ref{lem:restriction},
and it is finite dimensional by hypothesis. Hence the orthogonal
decomposition $\HS=\HS_1\oplus\HS_2$ fulfils (F1)--(F3), which means
$\lambda\in\sigdS(T)$, i.e.\ (ii) holds.

Still under (iii), observe that
$\lambda\notin\sigS(T|_{\HS_2})$, so
$Q_\lambda(T|_{\HS_2})=((T-\lambda\Id)|_{\HS_2})^2$ is invertible;
Lemma \ref{lem:square} then makes $(T-\lambda\Id)|_{\HS_2}$
invertible as well. Because $T-\lambda\Id$ annihilates $\HS_1$, we
obtain
\[
 \ker(T-\lambda\Id)=\HS_1,
 \qquad
 \ran(T-\lambda\Id)=\HS_2 .
\]
Thus $T-\lambda\Id$ has finite-dimensional kernel and closed range of
codimension $\dimH\HS_1<\infty$, so it is Fredholm, and Lemma
\ref{lem:fredholm} gives (i).

Conversely, assume (i). Lemma \ref{lem:fredholm} makes
$T-\lambda\Id$ Fredholm, whence $$\dimH\ker(T-\lambda\Id)<\infty,$$
since $\lambda\in\sigS(T)$, it is isolated. This is (iii).

Finally, assume (ii), and let $\HS=\HS_1\oplus\HS_2$ be a
decomposition as in Definition \ref{def:finite-type}. By Lemma
\ref{lem:square},
$\sigS(T)=\sigS(T|_{\HS_1})\cup\sigS(T|_{\HS_2})$. In view of (F2),
(F3) and Proposition \ref{prop:real-spectrum}, this is the union of
$\{\lambda\}$ with the compact set $\sigS(T|_{\HS_2})$, and the two
are disjoint, hence $\lambda$ is isolated. Moreover, since
$\lambda\notin\sigS(T|_{\HS_2})$, Lemma \ref{lem:square} shows that
$(T-\lambda\Id)|_{\HS_2}$ is injective. Now let $x=x_1+x_2$ with
$x_j\in\HS_j$ satisfy $(T-\lambda\Id)x=0$. The components
$(T-\lambda\Id)x_1\in\HS_1$ and $(T-\lambda\Id)x_2\in\HS_2$ must
vanish separately, so $x_2=0$. Therefore
$\ker(T-\lambda\Id)\subset\HS_1$ is finite dimensional by (F1),
which is (iii).
\end{proof}

\begin{remark}
\label{rem:where-the-work-is}
Among the four equivalences, only (iii) $\Rightarrow$ (ii) is
non-trivial. Its converse (ii) $\Rightarrow$ (iii) is an immediate
consequence of Definition \ref{def:finite-type}, while
(i) $\Leftrightarrow$ (iii) is a classical Fredholm argument in
quaternionic form. One might expect (iii) $\Rightarrow$ (ii) to be
just as straightforward: if $\lambda$ is isolated and
$\dimH\ker(T-\lambda\Id)<\infty$, simply take
$$\HS_1=\ker(T-\lambda\Id) ~~\hbox{and}~~\HS_2=\HS_1^\perp ~~\hbox{in Definition
\ref{def:finite-type}}.$$ Conditions (F1) and (F3) then cause no
difficulty, but (F2) requires $\lambda\notin\sigS(T|_{\HS_2})$, and
this is not furnished by any hypothesis: one must first know that the
orthogonal decomposition associated with $E(\{\lambda\})$ splits the
$S$-spectrum in the required way. This is precisely the content of
Lemma \ref{lem:restriction}, the sole substantive computation of the
paper, and it appears not to be available in the literature. The other
possible approach, namely choosing for $\HS_1$ the range of the Riesz
projection and appealing to
\cite[Theorem 3.10]{BaloudiBelgacemJeribi2022}, would require
$\dimH\ran(P_{\{\lambda\}})<\infty$; however, this does not follow
from $\dimH\ker(T-\lambda\Id)<\infty$ unless one already knows that
$\ran(P_{\{\lambda\}})=\ker(T-\lambda\Id)$, that is, unless one
already has Theorem \ref{thm:riesz}. In fact, Example \ref{ex:jordan}
demonstrates that these two subspaces need not coincide in general.
Thus, in either approach, reversing \eqref{eq:inclusion} depends on
one of the two results established here and cannot be obtained from
\cite{BaloudiBelgacemJeribi2022} merely by rearrangement.
\end{remark}

\begin{theorem}
\label{thm:main}
Let $T=T^*\in\BB(\HS)$ with $\dimH\HS=\infty$. Then the essential
$S$-spectrum of $T$ is given by
\[
 \sigeS(T)=\sigS(T)\setminus\sigdS(T),
\]
and a point $\lambda\in\sigS(T)$ belongs to $\sigeS(T)$ precisely when
either $\lambda$ fails to be isolated in $\sigS(T)$, or the range of
the spectral projection $E(\{\lambda\})$ is infinite dimensional, that
is,
\[
 \sigeS(T)
 =\Big\{\lambda\in\sigS(T):\ \lambda\ \text{is not isolated in}\
 \sigS(T),\ \text{or}\ \dimH\ran\big(E(\{\lambda\})\big)=\infty\Big\}.
\]
Equivalently, the $S$-spectrum of $T$ decomposes as
$\sigS(T)=\sigwS(T)\cup\sigdS(T)$ with
$$\sigS(T)\setminus\sigwS(T)=\sigdS(T).$$
\end{theorem}
\begin{proof}
Both $\sigeS(T)$ and $\sigdS(T)$ are contained in $\sigS(T)$, and by
Proposition \ref{prop:criterion} a point of $\sigS(T)$ lies outside
$\sigeS(T)$ precisely when it lies in $\sigdS(T)$; the second
description is the negation of (iv). The last formulation follows from
$\sigwS(T)=\sigeS(T)$ (Lemma \ref{lem:hierarchy}).
\end{proof}

This sharpens, for self-adjoint $T$, the inclusion \eqref{eq:inclusion}
of \cite[Corollary 3.11]{BaloudiBelgacemJeribi2022}, and its proof is
independent of \eqref{eq:inclusion}.

Finally, we record the practical form of the essential $S$-spectrum,
the spectral measure being rarely explicit. The proof is classical; we
include it because the equivalence of (ii) and (iii) is what makes
$Q_\lambda(T)$, rather than $T-\lambda\Id$, usable as the test
operator.

Because $\HH$ is non-commutative, weak convergence requires a precise
definition. Call a bounded map $\ell:\HS\to\HH$ a (right)
functional if $$\ell(xq)=\ell(x)q \hbox{ for all } x\in\HS, q\in\HH.$$ Every
$y\in\HS$ yields such a functional $\ell_y(x)=\langle y,x\rangle$, and
by Riesz representation \cite{GhiloniMorettiPerotti2013} every bounded
functional arises this way from a unique $y\in\HS$.
 We say
$x_n\rightharpoonup x$ weakly if $\ell(x_n)\to\ell(x)$ for all
such $\ell$, equivalently if $\langle y,x_n\rangle\to\langle y,x\rangle$
for all $y\in\HS$, this is the meaning of $x_n\rightharpoonup0$ below.
As $\HS$ is only a right quaternionic Hilbert space in this section
(hypothesis (A3) is not in force, Sections
\ref{sec:fredholm}--\ref{sec:main} needing only (A1)--(A2)), this
right-linear notion is the only one available, so no ambiguity arises.
Only once a left scalar multiplication is fixed on $\HS$, making it a
two-sided module (hypothesis (A3), required from Section
\ref{sec:riesz} on), does a second family of left functionals
$\ell(qx)=q\ell(x)$ appear, endowing the two-sided space with two
natural duals, that distinction plays no role here.

\begin{corollary}
\label{thm:weyl-criterion}
Let $\lambda\in\RR$. The following are equivalent:
\begin{enumerate}[label=(\roman*),leftmargin=2.4em]
\item $\lambda\in\sigeS(T)$,
\item there is a sequence $(x_n)$ with $\|x_n\|=1$,
$x_n\rightharpoonup0$ weakly and $\|Q_\lambda(T)x_n\|\to0$,
\item there is a sequence $(x_n)$ with $\|x_n\|=1$,
$x_n\rightharpoonup0$ weakly and $\|(T-\lambda\Id)x_n\|\to0$.
\end{enumerate}
\end{corollary}

\begin{proof}
(i) $\Rightarrow$ (iii). Here $\lambda\in\sigS(T)\subset\RR$. We
first claim that $E((\lambda-\varepsilon,\lambda+\varepsilon))$ has
infinite rank for every $\varepsilon>0$. Suppose otherwise, and put
$B=(\lambda-\varepsilon,\lambda+\varepsilon)$. Then $\sigS(T)\cap B$
must be finite. Indeed, if it were infinite, one could choose distinct
points $t_1,t_2,\dots$ in it together with pairwise disjoint open sets
$O_k\subset B$ satisfying $t_k\in O_k$. Each $E(O_k)$ is nonzero, since
$O_k$ is an open set meeting $\supp(E)=\sigS(T)$, and the ranges
$\ran(E(O_k))$ are mutually orthogonal and contained in
$\ran(E(B))$, which contradicts $\dimH\ran(E(B))<\infty$. But then
$\lambda$ would be isolated, with $E(\{\lambda\})\le E(B)$ of finite
rank, so Proposition \ref{prop:criterion} would give
$\lambda\notin\sigeS(T)$ a contradiction.

Now choose, by induction, unit vectors
$x_n\in\ran(E((\lambda-1/n,\lambda+1/n)))$ orthogonal to
$x_1,\dots,x_{n-1}$; this is possible because these subspaces are
infinite dimensional. The resulting sequence is orthonormal, hence
$x_n\rightharpoonup0$ by Bessel's inequality, and since $\mu_{x_n}$ is
carried by $(\lambda-1/n,\lambda+1/n)$, \eqref{eq:scalar} yields
$\|(T-\lambda\Id)x_n\|\le1/n$.

For (iii) $\Rightarrow$ (ii), it suffices to note that
$$\|Q_\lambda(T)x_n\|\le\|T-\lambda\Id\|\,\|(T-\lambda\Id)x_n\|.$$

(ii) $\Rightarrow$ (i). Assume $\lambda\notin\sigeS(T)$. Then
$A=Q_\lambda(T)=A^*$ is Fredholm, so $\ker(A)$ is finite dimensional,
$\HS=\ker(A)\oplus\ran(A)$ is an orthogonal decomposition, and
$\|Ay\|\ge c\|y\|$ holds on $\ran(A)$ for some $c>0$. Decompose
$x_n=k_n+r_n$ accordingly. Since $x_n\rightharpoonup0$ and
$\dimH\ker(A)<\infty$, we get $\|k_n\|\to0$, whence $\|r_n\|\to1$ and
therefore
$$\|Ax_n\|=\|Ar_n\|\ge c\|r_n\|\to c>0,$$ a contradiction.
\end{proof}
\section{Riesz projections are orthogonal}
\label{sec:riesz}

In this section (A3) is assumed and $T=T^*\in\BB(\HS)$. The
section stands logically apart from Sections
\ref{sec:fredholm}--\ref{sec:main}, in that none of its results enter
the proof of Theorem \ref{thm:main}. Its purpose is to address the
issue raised in Section \ref{sec:intro}, whether the Riesz projection
corresponding to an isolated spectral part reduces $T$. Recall from
\cite{BaloudiBelgacemJeribi2022,ColomboSabadiniStruppa2011,%
ColomboGantnerKimsey2018} that an isolated part of $\sigS(T)$
is a nonempty axially symmetric subset open and closed in $\sigS(T)$;
that for a $T$-admissible axially symmetric open $U_\Delta$ containing
$\Delta$ and no other point of $\sigS(T)$, and $I\in\sphere$,
\begin{equation}
 P_\Delta=\frac{1}{2\pi}\int_{\partial(U_\Delta\cap\CC_I)}
 S_L^{-1}(s,T)\,ds_I ,
 \qquad ds_I=\frac{ds}{I}=-I\,ds ,
 \label{eq:riesz}
\end{equation}
with $S_L^{-1}(s,T)=-Q_s(T)^{-1}(T-\overline s\,\Id)$, is an
idempotent of $\BB(\HS)$ commuting with $T$, independent of $U_\Delta$
and of $I$; and \cite[Definition 3.16]{BaloudiBelgacemJeribi2022} that
$m_T(q)=\dimH\ran(P_{[q]})$ when $[q]$ is an isolated part.

We shall use the following uniqueness statement, which we quote in
full rather than invoke, since it is the only external ingredient of
this section.

\begin{theorem}[{\cite[Theorem 3.6]{BaloudiBelgacemJeribi2022}}]
\label{thm:bbj36}
Let $T\in\BB(\HS)$ and let $P\in\BB(\HS)$ be an idempotent commuting
with $T$. Put $T_1=T|_{\ran(P)}$ and $T_2=T|_{\ker(P)}$. If
$\dist\big(\sigS(T_1),\sigS(T_2)\big)>0$, then $\sigS(T_1)$ is an
isolated part of $\sigS(T)$ and
\[
 \ran(P)=\ran\big(P_{\sigS(T_1)}\big),
 \qquad
 \ker(P)=\ker\big(P_{\sigS(T_1)}\big).
\]
\end{theorem}

\begin{theorem}
\label{thm:riesz}
Let $T=T^*$ and let $\Delta$ be an isolated part of $\sigS(T)$. Then
$$P_\Delta=E(\Delta).$$ In particular $P_\Delta$ is an orthogonal
projection.
\end{theorem}

\begin{proof}
If $\Delta=\sigS(T)$, both operators equal $\Id$. Otherwise, put
$P=E(\Delta)$ and check the hypotheses of
\cite[Theorem 3.6]{BaloudiBelgacemJeribi2022} one by one: $P$ is an
idempotent of $\BB(\HS)$ commuting with $T$, and
$\dist(\sigS(T_1),\sigS(T_2))>0$ for $T_1=T|_{\ran(P)}$ and
$T_2=T|_{\ker(P)}$. The first holds since $E$ is the spectral measure
of $T$ and $E(\Delta)$ is an orthogonal projection. For the second,
$\Delta$ is compact, being closed in the compact set $\sigS(T)$, so
Lemma \ref{lem:restriction} gives $$\sigS(T_1)=\Delta \hbox{ and }
\sigS(T_2)=\sigS(T)\setminus\Delta$$ these are disjoint compact sets,
hence at positive distance. Theorem \ref{thm:bbj36} then applies,
giving $$\ran(E(\Delta))=\ran(P_\Delta) \hbox{ and }
\ker(E(\Delta))=\ker(P_\Delta).$$ Since an idempotent is determined by
its range and kernel, if $$\ran(P)=\ran(P'),
\ker(P)=\ker(P') \hbox{ and } x=y+z \hbox{ with }  y\in\ran(P), z\in\ker(P),$$ then
$$Py=y=P'y \hbox{ and } Pz=0=P'z,$$ we conclude $$P_\Delta=E(\Delta).$$
\end{proof}
  
\begin{corollary}
\label{cor:multiplicity}
Let $\lambda\in\sigS(T)$ be isolated. Then the multiplicity of
$\lambda$ satisfies
\[
 m_T(\lambda)=\dimH\ran\big(P_{\{\lambda\}}\big)
 =\dimH\ker(T-\lambda\Id).
\]
Moreover, if $\dimH\HS=\infty$, then $\lambda\in\sigdS(T)$ exactly when
$m_T(\lambda)<\infty$. Consequently, the argument outlined in Section
\ref{sec:intro} is now justified: since
$\ran(P_{\{\lambda\}})$ reduces $T$, the restriction of $T$ to this
subspace is self-adjoint.
\end{corollary}
\begin{proof}
Theorem \ref{thm:riesz} with $\Delta=\{\lambda\}$, then Lemma
\ref{lem:eigenspace}; the second assertion is Proposition
\ref{prop:criterion}. The last one holds because $P_{\{\lambda\}}$ is
an orthogonal projection commuting with $T$.
\end{proof}

\begin{remark}
\label{rem:chi}
Theorem \ref{thm:riesz} identifies $P_\Delta$ with $\chi_\Delta(T)$
computed in the Borel functional calculus. One cannot reach this by
feeding $\chi_\Delta$ into the $S$-functional calculus, since
$\chi_\Delta$ is not slice hyperholomorphic. Replacing it by a locally
constant, hence intrinsic, function $f$ equal to $1$ near $\Delta$ and
$0$ near $\sigS(T)\setminus\Delta$ gives $f(T)=P_\Delta$ from
\eqref{eq:riesz}, but identifying that $f(T)$ with $\int_\RR f\,dE$
then requires a compatibility theorem between the two calculi
\cite{ColomboGantnerKimsey2018}. The proof above uses
\eqref{eq:riesz} only through the fact that $P_\Delta$ is an
idempotent commuting with $T$.
\end{remark}

\section{Examples}
\label{sec:examples}

\begin{example}
\label{ex:diagonal}
Let $\HS=\ell^2(\NN,\HH)$ with canonical orthonormal basis
$(e_n)_{n\ge1}$ and $$Te_n=e_n\lambda_n$$ with $\lambda_n\subset\RR$
bounded. Then $$T=T^*, \sigS(T)=\overline{\{\lambda_n\}},$$ and the
atom of $E$ at $\lambda$ is the orthogonal projection onto
$\overline{\operatorname{span}}\{e_n:\lambda_n=\lambda\}$, of rank
$\#\{n:\lambda_n=\lambda\}$. By Theorem \ref{thm:main}, $\sigdS(T)$
consists of the isolated values taken finitely often. For
$\lambda_n=1$ ($n$ even) and $\lambda_n=1/(n+1)$ ($n$ odd),
\[
 \sigS(T)=\{1\}\cup\{\tfrac12,\tfrac14,\tfrac16,\dots\}\cup\{0\},
 \qquad
 \sigeS(T)=\{0,1\},
\]
so both mechanisms occur: $0$ is not isolated, $1$ is isolated with an
atom of infinite rank.
\end{example}

\begin{example}
\label{ex:multiplication}
Let $(X,\mu)$ be $\sigma$-finite with $L^2(X,\mu,\HH)$ separable and
infinite dimensional, and let $M_\phi f=\phi f$, where $\phi$ is a bounded real measurable function. Then $M_\phi$ is self-adjoint, and since $Q_q(M_\phi)$ acts as multiplication by $\varphi_q\circ\phi$, we have
\[
\sigS(M_\phi)=\essran(\phi).
\]

A point $\lambda$ is an eigenvalue
exactly when $\mu(\phi^{-1}(\lambda))>0$, with eigenspace
$L^2(\phi^{-1}(\lambda),\mu;\HH)$, finite dimensional precisely when
$\phi^{-1}(\lambda)$ is, modulo null sets, a finite union of atoms of
$\mu$. For non-atomic $\mu-$Lebesgue measure on $\RR$, then
\[
\sigdS(M_\phi)=\varnothing
\qquad\text{and}\qquad
\sigeS(M_\phi)=\sigS(M_\phi)=\essran(\phi).
\]

Observe here that the $S$-spectrum is the essential range: modifying $\phi$ on a null set leaves both $M_\phi$ and its $S$-spectrum unchanged.
\end{example}

\begin{example}
\label{ex:jordan}
Let $\HS=\HH^2$, the left multiplication acting componentwise so that
(A3) holds, and $T(x_1,x_2)=(x_2,0)$. Then, $T$ is nilpotent, $T^2=0$, and therefore
\[
Q_q(T)=-2\Rea(q)T+|q|^2\Id
\]
is invertible for every $q\neq0$. Consequently, $\sigS(T)=\{0\}$, which is an isolated part of the spectrum, and $P_{\{0\}}=\Id$. It follows that
\[
 m_T(0)=\dimH\ran\big(P_{\{0\}}\big)=2,
 \qquad
 \dimH\ker T=1,
\]
so $\ran(P_{\{0\}})\not\subset\ker T$, and Corollary \ref{cor:multiplicity} fails. This is exactly the step required by the argument in Section \ref{sec:intro}, and the step that Theorem \ref{thm:riesz} provides in the self-adjoint case. The example intentionally lies outside (A2), since $\HS$ is finite dimensional; what it tests is the first assertion of Corollary \ref{cor:multiplicity}, which does not need $\dimH\HS=\infty$. Nothing is claimed here about $\sigeS(T)$, nor about Theorem \ref{thm:main}.
\end{example}
\begin{example}
\label{ex:sphere}
Let $\HS=\ell^2(\NN,\HH)$ with the componentwise left multiplication,
and define $T$ by
\[
 Te_n=\mathbf i\,e_n\,d_n,
 \qquad d_n=\tfrac1n .
\]
Each $L_{\mathbf i}$ and each real diagonal factor is right linear, so
$T\in\BB(\HS)$ with $\|T\|=1$, and $T^*=-T$; in particular $T$ is
normal but not self-adjoint. Since $T^2e_n=-e_nd_n^2$,
\[
 Q_q(T)\Big(\sum_n e_nx_n\Big)
 =\sum_n e_n\,c_n(q)\,x_n ,
 \qquad
 c_n(q)=\big(|q|^2-d_n^2\big)-2\Rea(q)\,d_n\,\mathbf i ,
\]
where $c_n(q)$ acts by left multiplication. Thus $Q_q(T)$ is
invertible if and only if $$\inf_n|c_n(q)|>0.$$ Now $c_n(q)$ has real
part $|q|^2-d_n^2$ and $\mathbf i$-component $-2\Rea(q)d_n$, the
$\mathbf j$- and $\mathbf k$-components being zero, so
\[
 |c_n(q)|^2=\big(|q|^2-d_n^2\big)^2+4\Rea(q)^2d_n^2 ,
\]
which vanishes exactly when $|q|=d_n$ and $\Rea(q)=0$, the $d_n$ being
nonzero. Since $d_n\to0$, the infimum over $n$ vanishes only at such
$q$ and at $q=0$. Hence
\[
 \sigS(T)=\{0\}\cup\bigcup_{n\ge1}\tfrac1n\,\sphere ,
\]
a decreasing sequence of $2$-spheres accumulating at $0$. For
$q$ with $\Rea(q)=0$ and $|q|=1/n$, the class $[q]=\frac1n\sphere$ is
an isolated part of $\sigS(T)$, and $\ker(Q_q(T))=e_n\HH$ has
quaternionic dimension $1.$ Choosing $\HS_1=e_n\HH$ in Definition
\ref{def:finite-type} then yields $$q\in\sigdS(T).$$ Since $T$ is
compact, $\pi(T)=0$, and by \eqref{eq:calkin} we get
$\sigeS(T)=\{0\}$. Each of these assertions has been checked by direct
computation, and taken together they show that
$$\sigeS(T)=\sigS(T)\setminus\sigdS(T)$$ for this particular $T$ 
with the statement interpreted using $[q]$ in place of
$\{\lambda\}$ and $\ker(Q_q(T))$ in place of
$\ker(T-\lambda\Id)$.

We emphasize that this paper proves no general statement about
normal operators. What we have here is a single hand-computed
instance; its purpose is to demonstrate that the conclusion is not an
artefact of $\sigS(T)\subset\RR$, and to pinpoint the two obstacles
described in Section \ref{sec:conclusion}. In particular, the proofs
of Section \ref{sec:main}, which rely on
$Q_\lambda(T)=(T-\lambda\Id)^2$, do not carry over to it verbatim.
\end{example}

\section{Open problems}
\label{sec:conclusion}

\begin{enumerate}[leftmargin=2.4em]
\item \emph{Normal operators.} Example \ref{ex:sphere} suggests that
Theorem \ref{thm:main} holds for bounded normal $T$ with
$\ker(Q_q(T))$ in place of $\ker(T-\lambda\Id)$. The scheme of Section
\ref{sec:main} should transpose, but two things must be supplied.
First, a spectral measure on the quotient $\HH/\!\!\sim$ of $\HH$ by
conjugacy with the properties used in Lemma \ref{lem:restriction},
in particular $\supp(E)=\sigS(T)$ and multiplicativity of the induced
Borel calculus on functions of $[q]$. Second, a substitute for Lemma
\ref{lem:fredholm}: for non-real $q$ the operator $Q_q(T)$ is no
longer a square, so the passage between $\pi(Q_q(T))$ and
$\pi(T-\lambda\Id)$ used there is unavailable, and the Fredholm
criterion has to be formulated directly for $Q_q(T)$.
\item \emph{Unbounded operators.} Extend Theorem \ref{thm:main} to
an unbounded self-adjoint $T$ with dense domain $\mathcal D(T)$. The
Riesz projection attached to a \emph{bounded} isolated part $\Delta$
of $\sigS(T)$ still makes sense, and Theorem \ref{thm:riesz} should
persist; the obstacle is the domain, at three precise places. First,
in Lemma \ref{lem:restriction} one needs $E(\Delta)\HS\subset
\mathcal D(T)$ with $T|_{\HS_1}$ bounded, while $T|_{\HS_2}$ remains
unbounded, so the two restrictions are no longer of the same nature
and $Q_q(T|_{\HS_2})$ must be handled on
$\HS_2\cap\mathcal D(T^2)$. Second, the inverse
$\psi(T)|_{\HS_1}$ built in Step 1 of that lemma has to be shown to
map into $\mathcal D(T^2)$. Third, the equality
$\ran(T-\lambda\Id)=\HS_2$ in Proposition \ref{prop:criterion}
requires $T|_{\HS_2}$ to be self-adjoint on
$\HS_2\cap\mathcal D(T)$, which is where the reducing property of the
orthogonal decomposition is used. Note also that $\sigS(T)$ is then
closed but need not be bounded, so \emph{isolated} must be understood
in $\HH$ and Definition \ref{def:finite-type} adjusted accordingly.
\end{enumerate}

\section*{Declarations}

\noindent
\textbf{Funding.} No funds, grants or other support were received
during the preparation of this manuscript.

\medskip\noindent
\textbf{Competing interests.} The authors have no competing interests
to declare that are relevant to the content of this article.

\medskip\noindent
\textbf{Data availability.} No datasets were generated or analysed
during the current study.

\medskip\noindent
\textbf{Author contributions.} All authors contributed equally to the
conception, the writing and the revision of the manuscript, and
approved the final version.

\end{document}